\documentclass[reqno, 12pt]{amsart}
\usepackage{amssymb}
\usepackage{amsfonts}
\usepackage{srcltx}
\usepackage{enumerate}
\usepackage{cite}

\newtheorem{theorem}{Theorem}[section]
\newtheorem{lemma}[theorem]{Lemma}
\newtheorem{proposition}[theorem]{Proposition}
\newtheorem{corollary}[theorem]{Corollary}
\theoremstyle{definition}
\newtheorem{definition}[theorem]{Definition}

\theoremstyle{remark}

\begin{document}
\title[Minimal displacement and extremal spaces]{The minimal displacement
and extremal spaces}
\author[K. Bolibok]{Krzysztof Bolibok}
\author[A. Wi\'{s}nicki]{Andrzej Wi\'{s}nicki}
\author[J. Wo\'{s}ko]{Jacek Wo\'{s}ko}

\begin{abstract}
We show that both separable preduals of \thinspace $L_{1}$ and non-type~I $%
C^{\ast }$-algebras are strictly extremal with respect to the minimal
displacement of $k$-Lipschitz mappings acting on the unit ball of a Banach
space. In particular, every separable $C(K)$ space is strictly extremal.
\end{abstract}

\subjclass[2010]{ Primary 47H09, 47H10; Secondary 46B20, 46L35.}
\keywords{Minimal displacement, Lipschitz mapping, Optimal retraction.}
\address{Krzysztof Bolibok, Institute of Mathematics, Maria Curie-Sk\l %
odowska University, 20-031 Lublin, Poland}
\email{bolibok@hektor.umcs.lublin.pl}
\address{Andrzej Wi\'{s}nicki, Institute of Mathematics, Maria Curie-Sk\l %
odowska University, 20-031 Lublin, Poland}
\email{a.wisnicki@umcs.pl}
\address{Jacek Wo\'{s}ko, Institute of Mathematics, Maria Curie-Sk\l odowska
University, 20-031 Lublin, Poland}
\email{jwosko@hektor.umcs.lublin.pl}
\date{}
\maketitle

\section{Introduction}

Throughout the paper $(X,\left\Vert \cdot \right\Vert )$ denotes a real or
complex infinite-dimen\-sion\-al Banach space. The notion of the minimal
displacement was introduced by K. Goebel in \cite{Goebel}. Let $C$ be a
bounded closed and convex subset of $X$ and $T:C\rightarrow C$ a mapping.
The minimal displacement of $T$ is the number
\begin{equation*}
d_{T}=\inf \left\{ \left\Vert x-Tx\right\Vert :x\in C\right\} .
\end{equation*}%
Goebel showed that if $T$ is $k$-Lipschitz then
\begin{equation*}
d_{T}\leq \left( 1-\frac{1}{k}\right) r\left( C\right) \qquad \text{for }%
k\geq 1,
\end{equation*}%
where $r\left( C\right) =\inf_{x\in X}\sup_{y\in C}\left\Vert x-y\right\Vert
$ denotes the Chebyshev radius of $C.$ There are some\ spaces and sets with $%
d_{T}=\left( 1-\frac{1}{k}\right) r\left( C\right) .$ The minimal
displacement characteristic of $X$ is a function
\begin{equation*}
\psi _{X}\left( k\right) =\sup \left\{ d_{T}:T:B_{X}\rightarrow B_{X},\ T\in
L\left( k\right) \right\} ,\ k\geq 1,
\end{equation*}%
where $B_{X}$ denotes the closed unit ball of $X$ and $L\left( k\right) $ is
the class of $k$-Lipschitz mappings. It is known that
\begin{equation*}
\psi _{X}\left( k\right) \leq 1-\frac{1}{k}
\end{equation*}%
for any space $X$ and the spaces with $\psi _{X}\left( k\right) =1-\frac{1}{k%
}$ are said to be extremal. Among extremal spaces are some sequence and
function spaces such as $c_{0},c,C\left[ 0,1\right] ,BC(\mathbb{R)},BC_{0}(%
\mathbb{R)}$ (see \cite{GMMV, Pia}).

Recently Bolibok \cite{Bolibok1} proved that in $\ell _{\infty },$%
\begin{equation*}
\psi _{\ell _{\infty }}\left( k\right) \geq \left\{
\begin{array}{ll}
\left( 3-2\sqrt{2}\right) \left( k-1\right) , & \text{if\ }1\leq k\leq 2+%
\sqrt{2}, \\
1-\frac{2}{k}, & \text{if }k>2+\sqrt{2}.%
\end{array}%
\right.
\end{equation*}%
It is still an open problem whether the space $\ell _{\infty }$ is extremal
with respect to the minimal displacement. Recall that $\ell _{\infty }$ is
isometric to $C(\beta \mathbb{N)},$ where $\beta \mathbb{N}$ is the Stone-%
\v{C}ech compactification of $\mathbb{N}$. In this note we show that every
separable $C(K)$ space, where $K$ is compact Hausdorff, is strictly extremal
(see Definition \ref{Def1}). This is a consequence of a more general Theorem %
\ref{predual}, which states that all separable preduals of \thinspace $L_{1}$
are (strictly) extremal. An analogous result holds in the non-commutative
case of separable non-type~I $C^{\ast }$-algebras.

\section{Results}

We begin by recalling the arguments which show that (a real or complex) $%
c_{0}$ is extremal (see \cite{Go, GMMV}). Fix $k\geq 1$ and define a mapping
$T:B_{c_{0}}\rightarrow B_{c_{0}}$ by%
\begin{equation*}
Tx=T(x_{1},x_{2},x_{3},...)=(1,k\left\vert x_{1}\right\vert \wedge
1,k\left\vert x_{2}\right\vert \wedge 1,...).
\end{equation*}%
It is clear that $T\in L(k)$ and for any $x=(x_{1},x_{2},x_{3},...)\in
B_{c_{0}},$ $\left\Vert Tx-x\right\Vert >1-\frac{1}{k},$ since the reverse
inequality implies $\left\vert x_{1}\right\vert \geq \frac{1}{k},$ $%
k\left\vert x_{1}\right\vert \wedge 1=1$ and, consequently, $\left\vert
x_{i}\right\vert \geq \frac{1}{k}$ for $i=1,2,3,...$, which contradicts $%
x\in c_{0}.$ Notice that the minimal displacement $d_{T}=1-\frac{1}{k}$ is
not achieved by $T$ at any point of $B_{c_{0}}.$ This suggests the following
definition.

\begin{definition}
\label{Def1}A Banach space $X$ is said to be strictly extremal if for every $%
k>1,$ there exists a mapping $T:B_{X}\rightarrow B_{X},$ $T\in L\left(
k\right) ,$ such that $\left\Vert Tx-x\right\Vert >1-\frac{1}{k}$ for every $%
x\in B_{X}.$
\end{definition}

It follows from the above that $c_{0}$ is strictly extremal. On the other
hand we have the following result.

\begin{proposition}
\label{Prop1}Suppose that $B_{X}$ has the fixed point property for
nonexpansive mappings (i.e., every nonexpansive mapping $S:B_{X}\rightarrow
B_{X}$ has a fixed point). Then for every $k$-Lipschitz mapping $%
T:B_{X}\rightarrow B_{X},$ $k\geq 1,$ there exists $x\in B_{X}$ such that $%
\left\Vert Tx-x\right\Vert \leq 1-\frac{1}{k}.$ In particular, $X$ is not
strictly extremal.
\end{proposition}

\begin{proof}
Let $T:B_{X}\rightarrow B_{X}$ be $k$-Lipschitz. Then $\frac{1}{k}T$ is
nonexpansive and consequently there exists $\left\Vert x\right\Vert \leq
\frac{1}{k}$ such that $Tx=kx.$ Hence $\left\Vert Tx-x\right\Vert
=(k-1)\left\Vert x\right\Vert \leq 1-\frac{1}{k}.$
\end{proof}

Proposition \ref{Prop1} applies to all uniformly nonsquare Banach spaces,
uniformly noncreasy spaces as well as to $\ell _{\infty }$.

In what follows we need the following observation.

\begin{lemma}
\label{lem1}Suppose that $Y$ is a subspace of a Banach space $X$ and there
exists an $m$-Lipschitz retraction $R:B_{X}\rightarrow B_{Y}.$ Then
\begin{equation*}
\psi _{X}\left( k\right) \geq \frac{1}{m}\psi _{Y}\left( \frac{k}{m}\right)
\end{equation*}%
for every $k\geq 1.$
\end{lemma}

\begin{proof}
Fix $\varepsilon >0$ and select a $k$-Lipschitz mappping $T:B_{Y}\rightarrow
B_{Y}$ such that $\left\Vert Ty-y\right\Vert >\psi _{Y}\left( k\right)
-\varepsilon $ for every $y\in B_{Y}.$ Define $\widetilde{T}%
:B_{X}\rightarrow B_{X}$ by putting $\widetilde{T}x=(T\circ R)x,$\ $x\in
B_{X}.$ Then
\begin{equation*}
\psi _{Y}\left( k\right) -\varepsilon <\left\Vert TRx-Rx\right\Vert
=\left\Vert RTRx-Rx\right\Vert \leq m\left\Vert \widetilde{T}x-x\right\Vert
\end{equation*}%
for every $x\in B_{X}.$ Notice that $\widetilde{T}$ is $km$-Lipschitz and
hence%
\begin{equation*}
\psi _{X}\left( km\right) \geq \frac{1}{m}(\psi _{Y}\left( k\right)
-\varepsilon ).
\end{equation*}%
This completes the proof since $\varepsilon $ is arbitrary.
\end{proof}

Recall that $Y$ is said to be a $k$-complemented subspace of a Banach space $%
X$ if there exists a (linear) projection $P:X\rightarrow Y$ with $\left\Vert
P\right\Vert \leq k.$ The well known\ Sobczyk theorem asserts that $c_{0}$
is $2$-complemented in any separable Banach space $X$ containing it.

\begin{proposition}
Let $X$ be a separable space which contains $c_{0}.$ Then%
\begin{equation*}
\psi _{X}\left( k\right) \geq \frac{1}{2}-\frac{1}{k},\ k\geq 1.
\end{equation*}
\end{proposition}

\begin{proof}
Let $P:X\rightarrow c_{0}$ be a projection with $\left\Vert P\right\Vert
\leq 2$ and define a retraction $R:X\rightarrow c_{0}$ by%
\begin{equation*}
(Rx)(i)=\left\{
\begin{array}{ll}
\medskip (Px)(i), & \text{if\ }\left\vert (Px)(i)\right\vert \leq 1, \\
\frac{(Px)(i)}{\left\vert (Px)(i)\right\vert }, & \text{if }\left\vert
(Px)(i)\right\vert >1.%
\end{array}%
\right.
\end{equation*}%
Then $R$ is $2$-Lipschitz and $R(B_{X})\subset R(B_{c_{0}}).$ It is enough
to apply Lemma \ref{lem1} since $c_{0}$ is extremal.
\end{proof}

The most interesting case is if $X$ contains a $1$-complemented copy of $%
c_{0}.$

\begin{theorem}
\label{predual}Let $X$ be a separable infinite-dimensional Banach space
whose dual is an $L_{1}(\mu )$ space over some measure space $(\Omega
,\Sigma ,\mu ).$ Then $\psi _{X}\left( k\right) =1-\frac{1}{k},$ i.e., $X$
is an extremal space.
\end{theorem}

\begin{proof}
It follows from the Zippin theorem (see \cite[Theorem 1]{Zi1}) that if $X$
is a separable infinite-dimensional real Banach space whose dual is an $%
L_{1}(\mu )$, then $X$ contains a $1$-complemented copy of $c_{0}.$ If $X$
is complex, consider its real part to obtain a $1$-Lipschitz retraction $%
R:X\rightarrow c_{0}.$ It is now enough to apply Lemma \ref{lem1}.
\end{proof}

By examining the proof of Lemma \ref{lem1} we conclude that every separable
predual of $L_{1}(\mu )$ is in fact strictly extremal. It is well known that
$C(K),$ the Banach space of scalar-valued continuous functions on the
Hausdorff compact space $K$, is a predual of some $L_{1}(\mu )$ (see, e.g.,
\cite{La}). Hence we obtain the following corollary.

\begin{corollary}
\label{cor1}Every separable infinite-dimensional $C(K)$ space, for some
compact Hausdorff space $K,$ is strictly extremal.
\end{corollary}

There exists an extensive literature regarding complemented subspaces of
Banach spaces. Our next simple observation is concerned with the connection
between the existence of a nonexpansive retraction and the fixed point
property.

\begin{proposition}
Suppose that $Y$ is a subspace of a Banach space $X$ and there exists a
nonexpasive retraction $R:X\rightarrow Y$ such that $R(B_{X})\subset B_{Y}.$
If $B_{X}$ has the fixed point property for nonexpasive mappings, then $%
B_{Y} $ has the fixed point property, too.
\end{proposition}

\begin{proof}
Suppose, contrary to our claim, that there exists a nonexpansive mapping $%
T:B_{Y}\rightarrow B_{Y}$ without a fixed point. Then the mapping $T\circ
R:B_{X}\rightarrow B_{X}$ is nonexpansive and fixed point free which
contradicts our assumption.
\end{proof}

It is well-known that $B_{\ell _{\infty }}$ has the fixed point property,
whereas $B_{c_{0}}$ does not have the fixed point property for nonexpansive
mappings (see, e.g., \cite{GoKi}). Hence we obtain another proof of the
well-known result that there is no nonexpansive retraction from $B_{\ell
_{\infty }}$ into $B_{c_{0}}$. In the same way, we conclude that there is no
nonexpansive retraction from $B_{\ell _{\infty }}$ into $B_{\hat{c}_{0}},$
where $\hat{c}_{0}$ denotes the space of scalar-valued sequences converging
to $0$ with respect to a Banach limit. Notice that $\hat{c}_{0}$ is a
strictly extremal subspace of $\ell _{\infty }$ of codimension one.

The space $C(K)$ of complex-valued continuous functions on $K$ forms a
commutative $C^{\ast }$-algebra under addition, pointwise multiplication and
conjugation. The remainder of this paper deals with a special class of $%
C^{\ast }$-algebras. Recall (see, e.g., \cite[Definition 1.5]{Ro}) that a $%
C^{\ast }$-algebra $\mathcal{A}$ is called type I if every irreducible $\ast
$-representation $\varphi :\mathcal{A}\rightarrow B(H)$ on a Hilbert space $%
H $ satisfies $K(H)\subset \varphi (\mathcal{A})$ ($\varphi $ is called
irreducible if $\varphi (\mathcal{A})$ has no invariant (closed linear)
subspaces other than $\{0\}$ and $H$). A fundamental example of non-type I
$C^{\ast }$-algebra is the CAR (canonical anti-commutation relations)
algebra, defined as follows. Identify $B(\ell _{2})$ with infinite matrices
and define $\mathrm{CAR}_{d}$ to be all $T\in B(\ell _{2})$ so that there
exist $n\geq 0$ and $A\in M_{2^{n}}$ (the space of all $2^{n}\times 2^{n}$
matrices over $\mathbb{C}$) such that%
\begin{equation*}
T=\left[
\begin{array}{cccc}
A &  &  &  \\
& A &  &  \\
&  & A &  \\
&  &  & \ddots%
\end{array}%
\right] .
\end{equation*}%
The CAR algebra is the norm-closure of $\mathrm{CAR}_{d}$.

\begin{theorem}[W. Lusky \protect\cite{Lu}]
{\label{Lusky}} Every separable $L_{1}$-predual space $X$ (over $\mathbb{C}$%
) is isometrically isomorphic to a $1$-complemented subspace of the CAR
algebra.
\end{theorem}

Combining a remark after Theorem \ref{predual} with Theorem \ref{Lusky} we
deduce that the CAR algebra is strictly extremal. It turns out that the same
is true for all separable non-type~I $C^{\ast }$-algebras.

\begin{theorem}
Every separable non-type~I $C^{\ast }$-algebra is strictly extremal.
\end{theorem}

\begin{proof}
It follows from \cite[Corollary 1.7]{Ro} that if $\mathcal{A}$ is a
separable non-type I $C^{\ast }$-algebra then the CAR algebra is isometric
to a $1$-complemented subspace of $\mathcal{A}$. It suffices to follow the
arguments of Lemma \ref{lem1} since the CAR algebra is strictly extremal.
\end{proof}

\bigskip

\bigskip

\end{document}